\documentclass[smallextended, envcountsect, envcountsame]{svjour3}
\usepackage{amsmath,amsfonts,ntheorem,enumerate,amssymb,textcomp}
\usepackage[mathscr]{euscript}
\usepackage[cp1250]{inputenc}
\usepackage[T1]{fontenc}

\def\={\discretionary{-}{-}{-}}

\newtheorem{ttheorem}{Theorem}
\newtheorem{wn}{Corollary}
\newtheorem{ddefinition}{Definition}
\theoremstyle{remark}
\newtheorem*{prob}{Problem}
\newtheorem*{que}{Question}

\authorrunning{Piotr Sworowski}
\institute{Piotr Sworowski \at Casimirus the Great University, Institute of Mathematics, Powsta\'nc\'ow Wielkopolskich 2, 85-090 Bydgoszcz, Poland, e-mail: {\tt piotrus@ukw.edu.pl}}
\title{An HK$_r$-integrable function which is P$_s$-integrable for no~$s$}
\begin{document}
\maketitle
\begin{abstract}
Given arbitrary $r\ge1$, we construct an HK$_r$-integrable function which is not P$_1$-integrable. This is an extension of a recently published Musial et al.\ construction [Musial, P., Skvortsov, V.,  Tulone, F.: The HK$_r$-integral is not contained in the P$_r$-integral. Proc. Amer. Math. Soc. {\bf150}(5), 2107--2114 (2022)].

\keywords{Riemann sum\and major/minor function\and HK$_r$-integral\and P$_r$-integral}

\subclass{26A39}
\end{abstract}

\bigskip
One of central topics in the generalized integration theory in the real line or, more generally, in~$\mathbb R^n$, is the correspondence between integral defined in the realm of integral sums (e.g., Riemann sums) and antidifferentiation. A standpoint result in this connection, valid in various settings (e.g., for differentiation bases), claims that generalized antidifferentiation (a Perron-type integral) and the integration w.r.t.\ Riemann sums (a Henstock--Kurzweil-type integral) are equivalent if defined correspondingly to each other. This result, however, need not extend to all possible settings, in particular to those where generalized derivates and integral sums are defined with respect to some integral means.

\medskip
Musial and Sagher in \cite{MusialSagher2004} made an attempt to construct a Henstock--Kurzweil-type integral (under the name {\em HK$_r$-integral}) which would cover the so-called P$_r$-integral, a Perron-type integral based on derivates defined via $L^r$ integral means~\cite{gordon}. This attempt wasn't successful in the sense that these two integrals turned out to be nonequivalent, a result announced recently in~\cite{MST1}. In the present note we demonstrate that this nonequivalence is actually much more dramatic and holds for all pairs of parameters $r,s\ge1$ (HK$_r$-integral vs.\ P$_s$-integral).
\section{Preliminaries}
\subsection*{HK$_r$-integral}
In order to make this presentation as compact as possible, we give up a detailed introduction to the $L^r$-Henstock--Kurzweil integral, which can be found e.g.\ in~\cite{MusialSagher2004}, and confine ourselves to its descriptive characterization as given in what follows.
\begin{ddefinition}\label{111}
A function $F\in L^r[a,b]$ is said to be {\it$L^r$-differentiable} at $x\in(a,b)$, if there exists a real number $\alpha$ such that
\[\left(\frac1h\int_{-h}^{h}\vert F(x+t)-F(x)-\alpha t\vert^r\mkern1.5mu dt\right)^{\!{1/r}}=o(h)\]
as $h\to0$. In this case we say that $\alpha$ is the {\it$L^r$-derivative} of $F$ at $x$ and write $F_r'(x)=\alpha$.
\end{ddefinition}
Obvious modifications of the above lead to the definition of $F_r'(a)$ and $F_r'(b)$.
\begin{ttheorem}\label{riwelb}
A function $f\colon[a,b]\to\mathbb R$ is $L^r$ Henstock--Kurzweil integrable (HK$_r$-integrable) iff there exists $F\in L^{r}[ a,b] $ such that for some $E\subset[a,b]$ of Lebesgue measure zero,
\begin{itemize}
\item $F'_r(x)=f(x)$ at all $x\in[a,b]\setminus E$;
\item $F$ has the AC$_r$ property on $E$, that is, to every $\varepsilon>0$ one can find a $\delta\colon E\to(0,\infty)$ such that if $[c_1,d_1],\dots,[c_n,d_n]$ is a collection of nonoverlapping subintervals in $[a,b]$ such that $[c_i,d_i]\subset(x_i-\delta(x_i),x_i+\delta(x_i))$ for some $x_i\in E\cap[c_i,d_i]$, $i=1,\dots,n$, then
$$\sum_{i=1}^n\left(\frac1{d_i-c_i}\int_{c_i}^{d_i}|F(y)-F(x_i)|^r\,dy\right)^{\!1/r}<\,\varepsilon.$$
\end{itemize}
\end{ttheorem}
The theorem easily follows from \cite[Theorem 14]{MusialSagher2004}.
\subsection*{P$_r$-integral}
A function $F\in L^r[a,b]$ is said to be {\it$L^r$-continuous} at $x\in(a,b)$ if
\[\left(\frac1h\int_{-h}^h\lvert F(x+t)-F(x)\rvert^r\mkern1.5mu dt\right)^{\!{1/r}}\!\to\,0\]
as $h\to0$. $L^r$-continuity of $F$ at $a$ and $b$ is understood accordingly.

Analogously like in Definition \ref{111} one can define upper and lower $L^r$-derivates. Namely, the {\em upper-right $L^r$-derivate} of $F\in L^r[a,b]$ at $x\in[a,b)$, denoted by $D_r^+F(x)$, is defined as the infimum of all numbers $\alpha$ such that
\begin{equation}\label{upper-right}
\left(\frac1h\int_0^h[F(x+t)-F(x)-\alpha t]_+^r\mkern1.5mu dt\right)^{\!1/r}\!=\,o(h)\tag{$\dagger$}
\end{equation}
as $h\rightarrow0^+$; here $[w]_+$ stands for $\max{\{w,0\}}$. If no real number $\alpha$ satisfies~\eqref{upper-right}, we set $D_r^+F(x)=+\infty$. The {\em lower-right $L^r$-derivate} of $F$ at $x$ is
$$D_{+,r}F(x)=-D_r^+(-F)(x).$$
The {\em upper-left} and {\em lower-left $L^r$-derivates} of $F$ at $x\in(a,b]$, $D_r^-F(x)$ and $D_{-,r}F(x)$, are defined in a similar manner. Bilateral lower and upper $L^r$-derivates $\underline{D}_{\mkern1mu r}F(x)$ and $\overline{D}_rF(x)$ at $x\in[a,b]$, are understood accordingly. A function $m\in L^r[a,b]$ is said to be an {\em$L^r$-minor function} of $f\colon[a,b]\to\mathbb R$ if
\begin{itemize}
\item $m(a)=0$,
\item $m$ is $L^r$-continuous everywhere in $[a,b]$,
\item $\overline{D}_{\mkern1mu r}m(x)\le f(x)$ at almost every $x\in[a,b]$,
\item $\overline{D}_{\mkern1mu r}m(x)<\infty$ nearly everywhere in $[a,b]$.
\end{itemize}
We say $M$ is an {\em$L^r$-major function} for $f$ if $-M$ is an $L^r$-minor function for $-f$. An $f$ is said to be {\em$L^r$-Perron integrable} (or {\em$P_r$-integrable}) if for every $\varepsilon>0$ there are an $L^r$-major function $M$ and an $L^r$-minor function for $f$ such that $M(b)-m(b)<\varepsilon$. Then one defines $\int_a^bf=\sup_mm(b)=\inf_MM(b)$ (this definition is correct since $M-m$ is always nondecreasing \cite[Theorem 6]{gordon}).
\section{The result}
In \cite{MusialSagher2004,MST1} it has been demonstrated what follows.
\begin{ttheorem}
Given $r\ge1$, each P$_r$-integrable function is HK$_r$-integrable, but there exists an HK$_r$-integrable function which is not P$_r$-integrable.
\end{ttheorem}
Below we adjust some parameters of the construction from \cite{MST1} in order to obtain a stronger result.
\begin{ttheorem}
Given $r\ge1$, there exists an HK$_r$-integrable function which is not P$_1$-integrable.
\end{ttheorem}
\begin{proof}
Denote with $P\subset[0,1]$ a Cantor-like symmetric set with contiguous intervals $u_n$ of rank $n$, $n\in\mathbb N$, of length
$$u_n=\frac{4^r-2}{4^{rn}},$$
where the numerator is chosen so that $P$ is a nullset:
$$\sum_{n=1}^\infty\frac{2^{n-1}}{4^{rn}}=\frac{1}{4^r-2}.$$
In each $u_n$ pick a concentric open segment $v_n$ so that
$$\frac{v_n}{u_n}=\frac1{3^{nr}}.$$
With these parameters we can proceed along the lines as in the construction of \cite[Theorem 3.1]{MST1}.

\medskip
We define an indefinite HK$_r$-integral $F\colon[0,1]\to\mathbb R$ as follows: we set $F=n$ on every $v_n$ and $0$ elsewhere. Using the same convention as in \cite{MST1}, we agree $F$ can be made smooth over every contiguous interval $u_n$ (by a suitable mollification around endpoints of $v_n$) in the way that no subsequent estimate is affected. Notice that $F\in L^r$ as
$$\int_0^1|F|^r=\sum_{n=1}^\infty n^r2^{n-1}v_n=(2^{2r-1}-1)\sum_{n=1}^\infty n^r\biggl(\frac2{12^r}\biggr)^{\!n}<\infty.$$
Next, we check that $F$ is indeed an indefinite HK$_r$-integral (of~$F'$), which boils down to checking that it has the AC$_r$ property on $P$ (Theorem \ref{riwelb}). Given $\varepsilon>0$ pick an $n$ such that
$$3^{1/r}\!\!\sum_{k=n+1}^\infty\!k\biggl(\frac23\biggr)^{\!k}<\varepsilon$$
and denote $\eta=(u_n-v_n)/2$. Consider any collection $\{I_i\}_{i=1}^q$ of nonoverlapping segments in $[0,1]$, $I_i\cap P\ne\emptyset$, such that $\sum_{i=1}^q|I_i|<\eta$. By the definition of $\eta$, $I_i\cap v_k=\emptyset$ for all $i$ and all $k\le n$. So, if $I_i\cap v_k\ne\emptyset$, then $k>n$ and
$$|I_i\cap u_k|>\frac{u_k-v_k}2=\frac{u_k}2\biggl(1-\frac1{3^{kr}}\biggr)\ge\frac{u_k}3,$$
and so
$$\frac1{|I_i\cap u_k|}\int_{I_i\cap u_k}\!|F|^r\le\frac3{u_k}\int_{u_k}\!|F|^r=\frac{3k^r}{3^{kr}}.$$
Thus, using a double decker inequality
$$\frac{\sum_ja_j}{\sum_jb_j}\le\sum_j\frac{a_j}{b_j},$$
which is valid for positive $a_j,b_j$, provided the left hand side is meaningful, we obtain
\begin{align}
\sum_{i=1}^q\left(\frac1{|I_i|}\int_{I_i}|F|^r\right)^{\!1/r}&\le\,\sum_{i=1}^q\left(\,\sum_{k:I_i\cap v_k\ne\emptyset}\frac1{|I_i\cap u_k|}\int_{I_i\cap u_k}\!\!|F|^r\right)^{\!1/r}\nonumber\intertext{(recall $P$ is a nullset). Then, we use the inequality $\sum_ja_j^{1/r}\ge\bigl(\sum_ja_j\bigr)^{1/r}$, valid for positive $a_j$, and continue with}
&\le\sum_{i,k:I_i\cap v_k\ne\emptyset}\left(\frac1{|I_i\cap u_k|}\int_{I_i\cap u_k}\!\!|F|^r\right)^{\!1/r}\nonumber\\
&\le\sum_{k=n+1}^\infty2\cdot2^{k-1}\biggl(\frac{3k^r}{3^{kr}}\biggr)^{\!1/r}=3^{1/r}\!\!\sum_{k=n+1}^\infty\!k\biggl(\frac23\biggr)^{\!k}<\varepsilon.\label{star}\tag{$\star$}
\end{align}
Thus $F$ is AC$_r$ on $P$.

\smallskip
Now, we are going to prove $F'$ is not P$_1$-integrable. Suppose not, then $F=(\textup{P}_1)\int F'$. There should exist an L$_1$-minor function $m$ for $F'$; then the difference $R=F-m$ will be a non-decreasing function on $[0,1]$ \cite[Theorem~6]{gordon}. Take any point $x\in P$ which is not right-hand isolated in~$P$. Given arbitrarily large~$N$, take an $n>N$ such that $x\in r_n\subset r_{n-1}$, where $r_n$ denotes any of the segments constituting the segment $[0,1]$ with all intervals $u_k$ up to rank $n$ removed. Note that
\begin{align*}
r_n&=\frac1{2^n}\left(1-\sum_{k=1}^n\frac{4^r-2}{4^{rk}}\cdot2^{k-1}\right)=\frac 1{4^{nr}}=\frac{u_n}{4^r-2}<u_n.
\end{align*}
Here we can assume $r_n$ is the left one of the two segments $r_n$ contained in $r_{n-1}$. Let $u_n$ be concentric with this $r_{n-1}$, take $h_n>0$ so that $x+h_n$ is the right endpoint of this $u_n$. We have
$$h_n\le r_n+u_n<2u_n.$$
Take arbitrary $\alpha\in\mathbb R$ and consider $n>R(1)+|\alpha|+1$; then $[m(x+t)-m(x)-\alpha t]_+=[F(x+t)-R(x+t)+R(x)-\alpha t]_+>1$ for all $x+t\in v_n\subset u_n$. Estimate
$$\frac1{h_n^2}\int_0^{h_n}[m(x+t)-m(x)-\alpha t]_+\,dt>\frac{v_n}{h_n^2}>\frac{v_n}{4u_n^2}=\frac1{4\mkern.6mu(4^r-2)}\cdot\biggl(\frac43\biggr)^{\!nr}.$$
This means, as $n$ could have been taken arbitrarily large, that $D_r^+m(x)=\infty$. Since $x$ here comes from a co-countable subset of~$P$, $m$ cannot be an L$^1$-minor function for~$F'$. Therefore $F'$ is not P$_1$-integrable.
\end{proof}
\begin{wn}
For every $r\ge1$ there is an HK$_r$-integrable function which is P$_s$-integrable for no $s$.
\end{wn}
Note that the $F$ constructed above, despite in $L^s$ for all $s<\infty$, is AC$_s$ on $P$ only if $s<r\log_23$ (see  \eqref{star}). Therefore, the following question remains open.
\begin{que}
Is there a function which is HK$_r$-integrable for all $r\ge1$, but not P$_1$-integrable?
\end{que}
\begin{prob}
Define a Perron-type integral equivalent to the HK$_r$-integral.
\end{prob}

\end{document}